%% file: ssvd.tex
\documentclass[11pt]{article}
\usepackage{amsmath, amsfonts, amssymb, amsthm,  graphicx, enumerate}
\usepackage[letterpaper, left=1truein, right=1truein, top = 1truein, bottom = 1truein]{geometry}

\usepackage{appendix}
\usepackage[numbers, square]{natbib}
\usepackage[
            CJKbookmarks=true,
            bookmarksnumbered=true,
            bookmarksopen=true,
            colorlinks=true,
            citecolor=red,
            linkcolor=blue,
            anchorcolor=red,
            urlcolor=blue
            ]{hyperref}
 \usepackage[usenames]{color}
 
 \newcommand{\blue}{\color{blue}}

\usepackage[ruled, vlined, lined, commentsnumbered]{algorithm2e}

\newcommand{\wh}{\widehat}
\newcommand{\wt}{\widetilde}

\renewcommand{\sp}{\mathrm{span}}

\newcommand{\thr}{\mathrm{thr}}
\newcommand{\mul}{\mathrm{mul}}

\newcommand{\sqnorm}[1]{\| #1 \|_{{\rm s}_q}}

\newcommand{\nb}[1]{{\sf\blue[#1]}}

\theoremstyle{plain}
\newtheorem{theorem}{Theorem}

\newtheorem{proposition}{Proposition}   
\newtheorem{lemma}{Lemma}

\theoremstyle{definition}

\newtheorem{condition}{Condition}

\newtheorem{remark}{Remark}

\newcommand{\dKL}{d_{\rm KL}}
\newcommand{\vol}{\mathrm{vol}}
\newcommand{\rmS}{{\rm S}}

\usepackage{prettyref,soul,xspace}
\input{defs.tex}

\title{Rate Optimal Denoising of Simultaneously Sparse and Low Rank Matrices
\footnote{Some partial results of this paper have been presented at the 2013 Allerton Conference as an invited paper \cite{Buja13}.}
}
\author{Dan Yang$^{1}$
Zongming Ma$^{2}$
and
Andreas Buja$^{2}$\\
\\
$^1$Rutgers University and $^2$University of Pennsylvania
}
\date{}

\begin{document}
	\maketitle

\begin{abstract}
We study minimax rates for
denoising simultaneously sparse and low rank matrices in high dimensions.
We show that an iterative thresholding algorithm achieves (near) optimal rates adaptively under mild conditions for a large class of loss functions.
Numerical experiments on synthetic datasets also demonstrate the competitive performance of the proposed method.

\bigskip\noindent
{\bf Keywords:}  Denoising, High dimensionality, Low rank matrices, Minimax rates, Simultaneously structured matrices, Sparse SVD, Sparsity.
\end{abstract}

\input{intro}

\input{method}

\input{theory}

\input{simulation}

\input{proof}

\input{appendix}

\bibliographystyle{plainnat}
\bibliography{refs}

\end{document}

%% file: defs.tex
\newcommand{\eg}{e.g.\xspace}
\newcommand{\ie}{i.e.\xspace}
\newcommand{\iid}{i.i.d.\xspace}

\newcommand{\reals}{{\mathbb{R}}}

\newcommand{\supp}{{\rm supp}}
\newcommand{\eexp}{{\rm e}}

\newcommand{\diff}{{\rm d}}

\newcommand{\rank}{\mathop{\sf rank}}

\newcommand{\diag}{\mathop{\text{diag}}}

\newcommand{\Expect}{\mathbb{E}}

\newcommand{\Prob}{\mathbb{P}}

\newcommand{\iiddistr}{{\stackrel{\text{\iid}}{\sim}}}

\newcommand{\Th}{{^{\rm th}}}

\newrefformat{eq}{(\ref{#1})}
\newrefformat{chap}{Chapter~\ref{#1}}
\newrefformat{sec}{Section~\ref{#1}}
\newrefformat{algo}{Algorithm~\ref{#1}}
\newrefformat{fig}{Fig.~\ref{#1}}
\newrefformat{tab}{Table~\ref{#1}}
\newrefformat{rmk}{Remark~\ref{#1}}
\newrefformat{clm}{Claim~\ref{#1}}
\newrefformat{def}{Definition~\ref{#1}}
\newrefformat{cor}{Corollary~\ref{#1}}
\newrefformat{lmm}{Lemma~\ref{#1}}
\newrefformat{lemma}{Lemma~\ref{#1}}
\newrefformat{prop}{Proposition~\ref{#1}}
\newrefformat{app}{Appendix~\ref{#1}}
\newrefformat{ex}{Example~\ref{#1}}
\newrefformat{exer}{Exercise~\ref{#1}}
\newrefformat{soln}{Solution~\ref{#1}}
\newrefformat{cond}{Condition~\ref{#1}}

\newcommand{\lunder}[1]{{\underset{\raise0.3em\hbox{$\smash{\scriptscriptstyle-}$}}{#1}}}

\newcommand{\norm}[1]{\|{#1} \|}

\newcommand{\fnorm}[1]{\|#1\|_{\rm F}}

\newcommand{\opnorm}[1]{\|#1\|_{\rm op}}

\newcommand{\Indc}{\mathbf{1}}

\def\innergetnumber#1[#2]#3{#2}
\def\getnumber{\expandafter\innergetnumber\jobname}

\newcommand{\bszero}{{\boldsymbol{0}}}

\newcommand{\calB}{{\mathcal{B}}}

\newcommand{\calF}{{\mathcal{F}}}

\newcommand{\calM}{{\mathcal{M}}}

\newcommand{\bfa}{{\mathbf{a}}}

\newcommand{\bfu}{{\mathbf{u}}}
\newcommand{\bfv}{{\mathbf{v}}}

\newcommand{\bfx}{{\mathbf{x}}}
\newcommand{\bfy}{{\mathbf{y}}}

\newcommand{\bfA}{{\mathbf{A}}}

\newcommand{\bfD}{{\mathbf{D}}}

\newcommand{\bfI}{{\mathbf{I}}}

\newcommand{\bfM}{{\mathbf{M}}}

\newcommand{\bfO}{{\mathbf{O}}}
\newcommand{\bfP}{{\mathbf{P}}}
\newcommand{\bfQ}{{\mathbf{Q}}}
\newcommand{\bfR}{{\mathbf{R}}}

\newcommand{\bfU}{{\mathbf{U}}}
\newcommand{\bfV}{{\mathbf{V}}}
\newcommand{\bfW}{{\mathbf{W}}}
\newcommand{\bfX}{{\mathbf{X}}}

\newcommand{\bfZ}{{\mathbf{Z}}}

\newcommand{\pth}[1]{\left( #1 \right)}
\newcommand{\qth}[1]{\left[ #1 \right]}
\newcommand{\sth}[1]{\left\{ #1 \right\}}

\newcommand{\KL}[2]{D(#1 \, || \, #2)}



%% file: intro.tex

\section{Introduction}
\label{sec:intro}

In recent years, there has been a surge of interest in 
estimating and denoising 
structured large matrices.
Leading examples include denoising low rank matrices \citep{donoho13},
recovering low rank matrices from a small number of entries, \ie, matrix completion \citep{Candes09, Candes10, Keshavan10,Koltchinskii11,Negahban11},
reduced rank regression \citep{Bunea11rank}, group sparse regression \citep{Yuan06, Lounici11}, among others.

In the present paper, we study the problem of denoising an $m\times n$ data matrix 
\begin{equation}
\label{eq:model}
\bfX = \bfM + \bfZ.
\end{equation}
The primary interest lies in the matrix $\bfM$ that is sparse in the sense that nonzero entries are assumed to be confined on a $k\times l$ block, which is not necessarily consecutive.
In addition to being sparse, the rank of $\bfM$, denoted by $r$, is assumed to be low.
Thus, $\bfM$ can be regarded as \emph{simultaneously structured} as opposed to those simply structured cases where $\bfM$ is assumed to be either only sparse or only of low rank.
To be concrete, we assume that $\bfZ$ consists of i.i.d.~additive Gaussian white noise with variance $\sigma^2$.
In the literature, the problem has also been referred to as the sparse SVD (singular value decomposition) problem. See, for instance, \cite{Yang11} and the references therein.

The interest in this problem is motivated by a number of related problems:
\begin{enumerate}
\item \emph{Biclustering}. 
It provides an ideal model for studying biclustering of microarray data. 
Let the rows of $\bfX$ correspond to cancer patients and the columns correspond to gene expression levels measured with microarrays.
A subset of $k$ patients can be clustered together as a subtype of the same cancer, which in turn is determined by a subset of $l$ genes. 
Moreover, the gene expression levels on such a bicluster can usually be captured by a low rank matrix. See, \eg, \citet{Shabalin09,Lee10,Butucea11,Sun13,Chen13}.

\item \emph{Recovery of simultaneously structured matrices with compressive measurements}.
There has been emerging interest in the signal processing community in recovering such simultaneously structured matrices based on compressive measurements, partly motivated by problems such as sparse vector recovery from quadratic measurements and sparse phase retrieval. See, \eg, \cite{Shechtman11,Li12} and the references therein.
The connection between the recovery problem and the denoising problem considered here is partially explored in \cite{Oymak13}. 
An interesting phenomenon in the recovery setting is that convex relaxation approach no longer works well \citep{Oymak12} as it does in the simply structured cases.


\item \emph{Sparse reduced rank regression}. The denoising problem is also closely connected to prediction in reduced rank regression where the coefficient matrix is also sparse. Indeed, let $n=l$, then problem \eqref{eq:model} reduces to sparse reduced rank regression with orthogonal design. See \cite{Bunea11} and \cite{Ma14} for more discussion.

\end{enumerate}

The main contribution of the present paper includes the following: 
i) We provide information-theoretic lower bounds for the estimation error of $\bfM$ under squared Schatten-$q$ norm losses for all $q\in [1,2]$;
ii) We propose a computationally efficient estimator that, under mild conditions, attains high probability upper bounds that match the minimax lower bounds within a multiplicative log factor (and sometimes even within a constant factor) simultaneously for all $q\in [1,2]$.
The theoretical results are further validated and supported by numerical experiments on synthetic data.


The rest of the paper is organized as follows.
In \prettyref{sec:method}, we precisely formulate the denoising problem and propose a denoising algorithm based on the idea of iterative thresholding.
\prettyref{sec:theory} establishes minimax risk lower bounds and high probability upper bounds that match the lower bounds within a multiplicative log factor for all squared Schatten-$q$ norm losses with $q\in [1,2]$.
\prettyref{sec:simulation} presents several numerical experiments which
demonstrate the competitive finite sample performance of the proposed denoising algorithm.
The proofs of the main results are presented in \prettyref{sec:proof},
with some technical details relegated to \prettyref{app:appendix}.

%% file: method.tex

\section{Problem Formulation and Denoising Method}
\label{sec:method}

\paragraph*{Notation} 
For any $a, b\in \reals$, let $a\wedge b = \min(a,b)$ and $a\vee b = \max(a,b)$. 
For any two sequences of positive numbers $\{a_n\}$ and $\{b_n\}$, we write $a_n = O(b_n)$ if $a_n\leq Cb_n$ for some absolute positive constant $C$ and all $n$.
For any matrix $\bfA\in \reals^{m\times n}$, denote its successive singular values by $\sigma_1(\bfA) \geq \dots \geq \sigma_{m\wedge n}(\bfA)\geq 0$.
For any $q\in [1,\infty)$, the Schatten-$q$ norm of $\bfA$ is defined as
$\sqnorm{\bfA} = (\sum_{i=1}^{m\wedge n}\sigma_i^q(\bfA))^{1/q}$.
Thus, $\norm{\bfA}_{\rmS_1}$ is the nuclear norm of $\bfA$ and $\norm{\bfA}_{\rmS_2} = \fnorm{\bfA}$ is the Frobenius norm.
In addition, the Schatten-$\infty$ norm of $\bfA$ is $\norm{\bfA}_{\rmS_\infty} = \sigma_1(\bfA) = \opnorm{\bfA}$, where $\opnorm{\cdot}$ stands for the operator norm.
The rank of $\bfA$ is denoted by $\rank(\bfA)$.
For any vector $\bfa$, we denote its Euclidean norm by $\norm{\bfa}$.
For any integer $m$, $[m]$ stands for the set $\{1,\dots, m\}$.
For any subset $I\subset [m]$ and $J\subset [n]$, we use $\bfA_{IJ}$ to denote the submatrix of $\bfA$ with rows indexed by $I$ and columns by $J$.
When either $I$ or $J$ is the whole set, we replace it with $*$. For instance, $\bfA_{I*} = \bfA_{I [n]}$.
Moreover, we use $\supp(\bfA)$ to denote the set of nonzero rows of $\bfA$.
For any set $A$, $|A|$ denotes its cardinality and $A^c$ denotes its complement.
A matrix $\bfA$ is called orthonormal, if the column vectors are of unit length and mutually orthogonal.
For any event $E$, we use $\Indc_E$ to denote the indicator function on $E$, and $E^c$ denotes its complement.

\subsection{Problem Formulation}
\label{sec:formulation}

We now put the denoising problem in a decision-theoretic framework.
Recall model \eqref{eq:model}.
We are interested in estimating $\bfM$ based on the noisy observation $\bfX$, 
where $\bfM$ is simultaneously sparse and low rank.
Let the singular value decomposition (SVD) of $\bfM = \bfU \bfD \bfV'$, where $\bfU$ is $m\times r$ orthonormal, $\bfV$ is $n\times r$ orthonormal and $\bfD = \diag(d_1,\dots, d_r)$ is $r\times r$ diagonal with $d_1\geq \cdots\geq d_r >0$.
In addition, since the nonzero entries on $\bfM$ concentrate on a $k\times l$ block, $\bfU$ has at most $k$ nonzero rows and $\bfV$ at most $l$.
Therefore, the parameter space of interest can be written as
\begin{equation}
\label{eq:para-space}
\begin{aligned}
\calF(m,n,k,l,r,d,\kappa) 
 = \{  & \bfM = \bfU \bfD \bfV'\in \reals^{m\times n}:
\rank(\bfM) = r,  \\
& |\supp(\bfU)| \leq k, |\supp(\bfV)| \leq l, \\
& d\leq d_r\leq \cdots \leq d_1 \leq \kappa  d  \}.
\end{aligned}
\end{equation}
We will focus on understanding the dependence of the minimax estimation error on the key model parameters $(m,n,k,l,r,d)$, while $\kappa > 1$ is treated as an unknown universal constant.
Without loss of generality, we assume $m\geq n$ here and after.
Note that it is implicitly assumed in \eqref{eq:para-space} that $m\geq k\geq r$ and $n\geq l\geq r$.

To measure the estimation accuracy, we use the following squared Schatten-$q$ norm loss functions:
\begin{equation}
\label{eq:loss}
L_q(\bfM, \wh\bfM) = \sqnorm{\wh\bfM - \bfM}^2, \qquad q\in [1,2].
\end{equation}
The model \eqref{eq:model}, the parameter space \eqref{eq:para-space} and the loss functions \eqref{eq:loss} give a precise formulation of the denoising problem.

\subsection{Approach}

From a matrix computation viewpoint, if one seeks a rank $r$ approximation to a matrix $\bfX$, then one can first find its left and the right $r$ leading singular vectors. If we organize these vectors as columns of the left and the right singular vector matrices $\bfU$ and $\bfV$, then the matrix $(\bfU\bfU') \bfX (\bfV \bfV')$ has the minimum Frobenius reconstruction error for $\bfX$ among all rank $r$ matrices, since $\bfU'\bfX\bfV$ will be a diagonal matrix consisting of the $r$ leading singular values of $\bfX$.
On the other hand, if one wants to enforce sparsity in the resulting matrix, it is natural to utilize the idea of thresholding in the above calculation.
Motivated by the above observation and also by an iterative thresholding idea previously used in solving sparse PCA problem \citep{Ma11,Yuan11}, 
we propose the denoising scheme in \prettyref{algo:IT} via two-way iterative thresholding.

\begin{algorithm}[tbh]
\SetAlgoLined
\KwIn{}
1. Observed data matrix $\bfX$. \\
2. Thresholding function $\eta$ and thresholds $\gamma_u$ and $\gamma_v$. \\
3. Rank~$r$ and noise standard deviation $\sigma$. \\
4. Initial orthonormal matrix $\bfV^{(0)} \in \reals^{n \times r}$. \\
\smallskip

\KwOut{Denoised matrix $\wh\bfM$.}
\smallskip

\Repeat{Convergence}{
\nl Right-to-Left Multiplication: $\bfU^{(t),\mul} = \bfX \bfV^{(t-1)}$. \\

\nl Left Thresholding: $\bfU^{(t),\thr} = (u^{(t),\thr}_{ij})$, with
$\bfU^{(t),\thr}_{i*}= \frac{\bfU^{(t),\mul}_{i*}}{\norm{\bfU^{(t),\mul}_{i*}}} \eta(\norm{\bfU^{(t),\mul}_{i*}},\gamma_u )$. 

\nl Left Orthonormalization with QR Decomposition: $\bfU^{(t)}\bfR_u^{(t)} = \bfU^{(t),\thr}$. \\

\nl Left-to-Right Multiplication: $\bfV^{(t),\mul} = \bfX' \bfU^{(t)}$. \\

\nl Right Thresholding: $\bfV^{(t),\thr} = (v^{(t),\thr}_{ij})$, with
$\bfV^{(t),\thr}_{i*}=\frac{\bfV^{(t),\mul}_{i*}}{\norm{\bfV^{(t),\mul}_{i*}}}\eta(\norm{\bfV^{(t),\mul}_{i*}},\gamma_v)$. \\

\nl Right Orthonormalization with QR Decomposition: $\bfV^{(t)}\bfR_v^{(t)} = \bfV^{(t),\thr}$.
}

\nl Compute projection matrices $\wh\bfP_u = \wh\bfU \wh\bfU'$ and $\wh\bfP_v = \wh\bfV\wh\bfV'$, where $\wh\bfU$ and $\wh\bfV$ are $\bfU^{(t)}$ and $\bfV^{(t)}$ at convergence.

\nl Compute denoised matrix $\wh\bfM = \wh\bfP_u \bfX \wh\bfP_v$.
\caption{Matrix Denoising via Two-Way Iterative Thresholding}
\label{algo:IT}
\end{algorithm}

Without the two thresholding steps, the iterative part of the algorithm computes the leading singular vectors of any rectangular matrix, and can be viewed as a two-way generalization of the power iteration \citep{GolubLoan96}. 

In the thresholding steps, we apply row-wise thresholding to the matrix $\bfU^{(t),\mul}$ (resp.~$\bfV^{(t),\mul}$) obtained after the multiplication step.
In the thresholding function $\eta(x,t)$, the second argument $t > 0$ is called the threshold level. 
In \prettyref{algo:IT}, the first argument $x$ is always non-negative.
In order for the later theoretical results to work, we impose the following minimal assumption on the thresholding function $\eta$:
\begin{equation}
\label{eq:eta}
\begin{aligned}
|\eta(x,t) - x| & \leq t,\quad \mbox{for any $x\geq 0$, $t > 0$,} \\
\eta(x,t)&  = 0,\quad \mbox{for any $t > 0$, $x\in [0,t]$.}
\end{aligned}
\end{equation}
Examples of such thresholding functions include the usual soft and hard thresholding, the SCAD \citep{FanLi01}, the MCP \citep{Zhang10}, etc.
Thus, for instance, when thresholding $\bfU^{(t),\mul}$, if $\eta$ is the hard thresholding function, then we are going to keep all the rows whose norms are greater than $\gamma_u$ and kill all the rows whose norms are smaller than $\gamma_u$.
For other thresholding function, we shrink the norms according to $\eta$ while keeping the phases of the row vectors.
Throughout the iterations, the threshold levels $\gamma_u$ and $\gamma_v$ are pre-specified and remain unchanged.
In order for the theorem to work, these levels can be chosen as in \eqref{eq:thr} below.

To determine the convergence of the iterative part, we could either run a pre-specified number of iterations or stop after the difference between successive iterates are sufficiently small, \eg, 
\begin{align}
	\label{eq:terminate}
\fnorm{\bfU^{(t)}(\bfU^{(t)})' - \bfU^{(t-1)}(\bfU^{(t-1)})'}^2 \vee
\fnorm{\bfV^{(t)}(\bfV^{(t)})' - \bfV^{(t-1)}(\bfV^{(t-1)})'}^2 \leq \epsilon,
\end{align}
where $\epsilon$ is a pre-specified tolerance level.

\paragraph*{Initialization}
To initialize \prettyref{algo:IT}, we need to further specify the rank $r$, the noise standard deviation $\sigma$ and a starting point $\bfV^{(0)}$ for the iteration.
For the ease of exposition, we assume that $r$ is known. Otherwise, it can be estimated by methods such as those described in \cite{Yang11}.
When we have Gaussian noise and $kl < \frac{1}{2}mn$, the noise standard deviation can be estimated by
\begin{equation}
\label{eq:noise-sd}
\wh\sigma = 1.4826\cdot \mathrm{MAD}(\{\bfM_{ij}:i\in [m],j\in [n] \}).
\end{equation}
Finally, to obtain a reasonable initial orthonormal matrix $\bfV^{(0)}$, we propose to use \prettyref{algo:init} for the case of Gaussian noise.

\begin{algorithm}[!tbh]
\SetAlgoLined
\KwIn{}
1. Observed data matrix $\bfX$. \\
2. Tuning parameter $\alpha$. \\
3. Rank~$r$ and noise standard deviation $\sigma$. \\
\smallskip

\KwOut{Estimators $\wh\bfU = \bfU^{(0)}$ and $\wh\bfV = {\bfV}^{(0)}$.}
\smallskip

\nl Select the subset $I_0$ of rows and the subset $J_0$ of columns as \label{algo:step:selection}
\begin{subequations}
\label{eq:set-I0J0}
\begin{equation}
\label{eq:set-I0}
I^0=\{i:\|\bfX_{i*}\|^2\geq \sigma^2(n + \alpha\sqrt{n\log n})\},
\end{equation}
\begin{equation}
\label{eq:set-J0}
J^0=\{j:\|\bfX_{* j}\|^2\geq \sigma^2(m + \alpha\sqrt{m\log m})\}.
\end{equation}
\end{subequations}

\smallskip

\nl Compute $\bfX^{(0)}=(x_{ij}^{(0)})$, where $x_{ij}^{(0)}=x_{ij}\Indc_{i\in I^0} \Indc_{j\in J^0}$. \\\label{algo:step:svd}
    
\nl Compute $\bfU^{(0)}=[\bfu^{(0)}_1,\dots,\bfu^{(0)}_r]$ and $\bfV^{(0)}=[\bfv^{(0)}_1,\dots,\bfv^{(0)}_r]$, where $\bfu_\nu^{(0)}$ ($\bfv_\nu^{(0)}$) is the $\nu\Th$ leading left (right) singular vector of $\bfX^{(0)}$.\\

\caption{Initialization for \prettyref{algo:IT}}
\label{algo:init}
\end{algorithm}

\begin{remark}
In practice,
\prettyref{algo:IT} and \prettyref{algo:init} are not restricted to the denoising of matrices with Gaussian noise.
With proper modification and robustification, they can be used together to deal with other noise distributions and/or outliers. See, \eg, \cite{Yang11}.
\end{remark}

%% file: theory.tex

\section{Theoretical Results}
\label{sec:theory}

In this section, we present a minimax theory underlying the denoising/estimation problem formulated in \prettyref{sec:formulation}.

\subsection{Minimax Lower Bounds}
\label{sec:lower}

\begin{theorem}
\label{thm:lower}
Let $\calF = \calF(m,n,k,l,r,d,\kappa)$ with $\kappa > 1$ and $k\wedge l\geq 2r$.
There exists a positive constant $c$ that depends only on $\kappa$, such that for any $q\in [1,2]$, the minimax risk for estimating $\bfM$ under the squared Schatten-$q$ error loss \eqref{eq:loss} satisfies
\begin{equation*}
	\label{eq:lowbd}
\begin{aligned}
& \inf_{\wh\bfM} \sup_{\calF} \Expect L_q(\bfM,\wh\bfM)
\geq 
c\sigma^2  \qth{\pth{r^{\frac{2}{q}-1} \frac{d^2}{\sigma^2} } 
\wedge \Psi_q(m,n,k,l,r)}
\end{aligned}
\end{equation*}
where the rate function 
$\Psi_q(m,n,k,l,r) = r^{\frac{2}{q}} (k+l) + r^{\frac{2}{q}-1}\pth{ k\log\frac{\eexp m}{k} + l\log\frac{\eexp n}{l}}$.
\end{theorem}

A proof of the theorem is given in \prettyref{sec:proof-lower}.

\begin{remark}
Regardless of the value of $q$, the lower bounds reflect two different scenarios.

The first scenario is the ``low signal'' case where
\begin{align}
	\label{eq:low-signal}
d^2 \leq \sigma^2 \Psi_2(m,n,k,l,r).
\end{align}
In this case, the first term in the lower bound \eqref{eq:lowbd} dominates, and the rate is achieved by simply estimating $\bfM$ by $\bszero\in \reals^{m\times n}$.

The second scenario is when \eqref{eq:low-signal} does not hold. 
In this case, the second term in \eqref{eq:lowbd} dominates.
We note this term is expressed as the sum of two terms. 
As to be revealed by the proof, the first summand is an ``oracle'' error term which occurs even when the indices of the nonzero rows and columns of $\bfM$ are given by an oracle.
In contrast, the second summand results from the combinatorial uncertainty about the locations of these nonzero rows and columns.
\end{remark}

\subsection{Minimax Upper Bounds}
\label{sec:upper}

To state the upper bounds, we first specify the threshold levels used in \prettyref{algo:IT}.
In particular, for some sufficiently large constant $\beta > 0$, set
\begin{align}
\label{eq:thr}
\gamma_u^2 = \gamma_v^2 = \gamma^2 = 1.01(r + 2\sqrt{r\beta \log{m}} + 2\beta\log{m}).
\end{align}
For \prettyref{thm:upper} to hold, it suffices to choose any $\beta\geq 4$.
In addition, we specify the stopping (convergence) rule for the loop in \prettyref{algo:IT}.
For $\bfX^{(0)}$ defined in \prettyref{algo:init},
let $d_r^{(0)}$ be its $r\Th$ largest singular value.
Define
\begin{align}
	\label{eq:whT}
\wh{T} = \frac{1.1}{2}\qth{\frac{\log{m}}{\log{2}} + \log\frac{(d_r^{(0)})^2}{\gamma^2}},
\end{align}
and
\begin{align}
\label{eq:T}
T = \frac{1.01}{2}\qth{\frac{\log{m}}{\log{2}} + \log\frac{d_r^2}{k\gamma_u^2 \vee l\gamma_v^2}}.
\end{align}
We propose to stop the iteration in \prettyref{algo:IT} after $\wh{T}$ steps.
Last but not least, we need the following technical condition.
\begin{condition}
\label{cond:asymp}
There exists a sufficiently small absolute constant $c$, such that 
$m\geq n$, 
$\log{d}\leq cm$, 
$c\leq \log{m}/\log{n} \leq 1/c$, 
$\log{m} \leq c [(m-k)\wedge(n-l)]$, 
$k\vee l \leq c(m\wedge n)$.
In addition, there exists a sufficiently small constant $c'$ that depends only on $\kappa$, such that
$d^{-2}{r}\pth{k\sqrt{n\log{m}} + l\sqrt{m\log{m}}} \leq c'$.
\end{condition}

With the above definition, the following theorem establishes high probability upper bounds of the proposed estimator.

\begin{theorem}
\label{thm:upper}
Let \prettyref{cond:asymp} be satisfied.
In \prettyref{algo:IT}, let $\bfV^{(0)}$ be obtained by \prettyref{algo:init} with $\alpha\geq 4$ in \eqref{eq:set-I0J0}.
Let $\gamma_u$ and $\gamma_v$ be defined as in \eqref{eq:thr} with $\beta \geq 4$.
Moreover, we stop the iteration after $\wh{T}$ steps with $\wh{T}$ defined in \eqref{eq:whT}, and use $\wh{\bfU} = \bfU^{(\wh{T})}$ and $\wh{\bfV}=\bfV^{(\wh{T})}$ in subsequent steps. For sufficiently large values of $m$ and $n$,
uniformly over $\calF(m,n,k,l,r,d,\kappa)$, with probability at least $1-O(m^{-2})$, $\wh{T}\in [T, 3T]$ and
\begin{align*}
\sqnorm{\wh{\bfM} - \bfM}^2 \leq C \sigma^2 \qth{r^{2\over q}(k+l+\log{m}) + r^{\frac{2}{q}-1} (k+l)\log{m}}
\end{align*}
where $C$ is a positive constant that depends only on $\kappa$ and $\beta$.
\end{theorem}

The proof of the theorem is given in \prettyref{sec:proof-upper}.

\begin{remark}
Under \prettyref{cond:asymp}, for sufficient large values of $m$ and $n$, \eqref{eq:low-signal} cannot hold, and so the relevant lower bound is $c\sigma^2 \Psi_q(m,n,k,l,r)$.
In comparison, when $k\wedge l\geq (1+\epsilon) r$ for any universal small constant $\epsilon > 0$, the upper bounds in \prettyref{thm:upper} always matches the lower bounds for all $q\in [1,2]$ up to a multiplicative log factor.
If in addition, $\log{m} = O(k\vee l)$ and $k = O(m^{a})$ and $l=O(n^a)$ for some constant $a \in (0,1)$, then the rates in the lower and upper bounds match exactly for all $q\in [1,2]$.
\end{remark}

\begin{remark}
The proposed estimator is adaptive since it does not depend on the knowledge of $k,l$ and $q$. 
Its dependence on $r$ can also be removed, as we explain in the next subsection.
\end{remark}

\subsection{Rank Selection}


We now turn to data-based selection of the rank $r$. 
Recall the sets $I^0$ and $J^0$ defined in \eqref{eq:set-I0J0}. 
We propose to use the following data-based choice of $r$:
\begin{align}
	\label{eq:r-hat}
	\wh{r} = \max\sth{s: \sigma_s(\bfX_{I^0 J^0}) \geq \sigma\, \delta_{|I^0| |J^0|}},
\end{align}
where for any $i\in [m]$ and $j\in [n]$, 
$\delta_{ij} = \sqrt{i}+\sqrt{j}+\sqrt{2i\log\frac{\eexp m}{i} + 2j\log\frac{\eexp n}{j} + 8\log{m}}$.
We note that it is straightforward to incorporate this rank selection step into \prettyref{algo:init}.
Indeed, we can compute $\wh{r}$ right after step 1 and replace all $r$ in the subsequent steps by $\wh{r}$.
The following result justifies our proposal.

\begin{proposition}
\label{prop:r-hat}
Under the condition of \prettyref{thm:upper}, $\wh{r} = r$ holds with probability at least $1-O(m^{-2})$.
\end{proposition}

A proof of the proposition is given in \prettyref{sec:proof-r-hat}. 
According to \prettyref{prop:r-hat}, we can use $\wh{r}$ as the input for rank in \prettyref{algo:IT} and the conclusion of \prettyref{thm:upper} continues to hold.

%% file: simulation.tex

\section{Simulation}
\label{sec:simulation}

In this section, we demonstrate the performance of the proposed denoising method on synthetic datasets.

In the first numerical experiment, we fix $m = 2000, n=1000, k = l = 50$ and $r = 10$.
On the other hand, we set the singular values of $\bfM$ as $(d_1,\dots, d_{10}) = a\times (200,190,\dots,120,110)$, where $a \in \sth{0.5, 1, 5, 10, 20}$.
The $\bfU$ matrix is obtained by orthonormalizing a $m\times r$ matrix the $i\Th$ row of which is filled i.i.d.~$N(0,i^4)$ entries for any $i\in [k]$ and  zeros otherwise.
The $\bfV$ matrix is obtained in the same way with $m$ and $k$ replaced by $n$ and $l$.
The noise standard deviation is set at $\sigma = 1$.
\prettyref{tab:sval} reports the average values of $L_q(\bfM, \wh\bfM)$ for $q = 2$ and $1$ and their standard errors out of $100$ repetitions for each value of $a$.
Throughout, we use \eqref{eq:noise-sd} to estimate $\sigma$, \prettyref{algo:init} with $\alpha = 4$ to compute $\bfV^{(0)}$ and \eqref{eq:r-hat} to select the rank.
In \prettyref{algo:IT}, we set $\beta = 3$ and we terminate the iteration once \eqref{eq:terminate} holds with $\epsilon = 10^{-10}$.
The thresholding function $\eta$ is fixed to be hard thresholding $\eta(x,t) = x\Indc_{|x| > t}$.
In all the repetitions, the proposed $\wh{r}$ in $\eqref{eq:r-hat}$ consistently yields the right rank $r = 10$.
From the results in \prettyref{tab:sval}, we conclude that the reconstruction error is stable across different choices of the singular values of $\bfM$, which agrees well with the theoretical results in \prettyref{thm:upper}.
We note that the magnitude of the average errors reported in \prettyref{tab:sval} is also expected. 
For reference, under the simulation setting, the oracle risk term for the Schatten-$2$ norm, modulo a constant factor, should be $\sigma^2 r (k+l) = 1000$, and for the Schatten-$1$ norm, modulo a constant factor, should be $\sigma^2 r^2 (k+l) = 10000$.

\begin{table}[!tb]
\begin{center}
	\smallskip
\begin{tabular}{|c|ccccc|}
\hline
$a$ & 0.5 & 1 & 5 & 10 & 20\\
\hline
\hline
Average($L_2(\bfM, \wh\bfM)$) & 
 1093.18 &  924.90 &  936.82 &  927.88 & 944.08\\
Standard error & 
(7.96) & (5.41) & (5.69) & (5.30) & (6.51)\\
\hline
Average($L_1(\bfM, \wh\bfM)$) &
18346.20 & 15993.79 & 16354.86 & 16277.88 & 16526.22\\
Standard error &
(115.06) & (84.82) & (95.22) & (89.57) & (104.87)\\
\hline
\end{tabular}
\caption{Average losses (and its standard error) of $\wh{\bfM}$ out of $100$ repetitions for different choices of singular values. 
\label{tab:sval}}
\end{center}
\end{table}

In the second experiment, we fix $m = 2000, n=1000, r = 10$ and the singular values of $\bfM$ are $(d_1,\dots, d_{10}) = (200,190,\dots,120,110)$.
On the other hand, we consider four different combinations of sparsity parameters: $(k,l) = (50,50), (50,200), (100, 200)$ and $(100, 50)$.
For each $(k,l)$ pair, the way we generate $\bfU$, $\bfV$ and $\bfX$ is the same as that in the first experiment.
Moreover, the tuning parameter values used in denoising are also the same as before.
In all the repetitions, $\wh{r}$ in \eqref{eq:r-hat} consistently select $r = 10$.
In \prettyref{tab:kl}, we report the average values of $L_q(\bfM, \wh\bfM)$ for $q = 2$ and $1$ and their standard errors over $100$ repetitions.
Moreover, we report the rescaled average loss where the rescaling constant is chosen to be $r^{\frac{2}{q}-1}(r+\log{m})(k+l)$, the rate derived in \prettyref{thm:upper}.
By the results reported in \prettyref{tab:kl}, we see that for either loss function, the rescaled average losses are stable with respect to different sparsity levels specified by different values of $k$ and $l$.
Again, this agrees well with the earlier theoretical results.

\begin{table}[!tb]
\begin{center}
	\smallskip
\begin{tabular}{|c|cccc|}
\hline
$(k,l)$ &  (50, 50) & (50, 200) & (100, 200) & (100, 50)\\
\hline
\hline
Average($L_2(\bfM, \wh\bfM)$) & 
1133.03 & 2662.07 & 3598.69 & 1673.49\\
Standard error & 
(5.96) & (11.73) & (12.84) & (9.73) \\
Average$\pth{\frac{L_2(\bfM, \wh\bfM)}{(r+\log{m})(k+l)}}$ &
0.64 & 0.60 & 0.68 & 0.63 \\
\hline
Average($L_1(\bfM, \wh\bfM)$) &
19056.47 & 43035.95 & 65099.19 & 28347.12\\
Standard error &
(88.42) & (172.39) & (231.98) & (146.07) \\
Average$\pth{\frac{L_1(\bfM, \wh\bfM)}{(r^2+ r\log{m})(k+l)}}$ &
1.08 & 0.98 & 1.23 & 1.07 \\
\hline
\end{tabular}
\caption{Average losses (with its standard error) and average rescaled losses of $\wh{\bfM}$ out of $100$ repetitions for different sparsity levels. 
\label{tab:kl}}
\end{center}
\end{table}

%% file: proof.tex

\section{Proofs}
\label{sec:proof}

\subsection{Proof of \prettyref{thm:lower}}
\label{sec:proof-lower}

\begin{proof}[Proof of \prettyref{thm:lower}]
To establish the lower bound, first consider the subset $\calF_1\subset \calF(m,n,k,l,r,d,\kappa)$ where we further require $\supp(\bfV) = [r]$.
Thus, except for the first $r$ columns, all columns of $\bfM$ are zeros.
So, by a simple sufficiency argument, we may assume that $n=l=r$.
In this case, the problem of estimating $\bfM$ under model \eqref{eq:model} can be viewed as a special case of sparse reduced rank regression where the design matrix is the identity matrix $\bfI_m$.
Therefore, \cite[Theorem 2]{Ma14} implies that 
\[
\inf_{\wh{\bfM}}\sup_{\calF} \Expect L_q(\bfM,\wh\bfM) \geq 
\inf_{\wh{\bfM}}\sup_{\calF_1} \Expect L_q(\bfM,\wh\bfM) \geq
c\qth{r^{\frac{2}{q}-1}d^2 \wedge \pth{r^{2\over q} k + r^{\frac{2}{q}-1}k\log\frac{\eexp m}{k} }}.
\]
By symmetry, we also have
\[
\inf_{\wh{\bfM}}\sup_{\calF} \Expect L_q(\bfM,\wh\bfM) \geq
c\qth{r^{\frac{2}{q}-1}d^2 \wedge \pth{r^{2\over q} l + r^{\frac{2}{q}-1}l\log\frac{\eexp n}{l} }}.
\]
We complete the proof by noting that for any $a,b,c>0$, $(a\wedge b)\vee (a\wedge c) = a\wedge (b\vee c) \asymp a\wedge (b+c)$.
\end{proof}

\subsection{Proof of \prettyref{thm:upper}}
\label{sec:proof-upper}

To prove \prettyref{thm:upper}, we follow the oracle sequence approach developed in \cite{Ma11}.
Throughout the proof, we assume that $\sigma = 1$ is known. 
The case of general $\sigma>0$ comes from obvious scaling arguments.
In what follows, we first define the oracle sequence and introduce some preliminaries. 
Then we give an overview of the proof, which is divided into three steps.
After the overview, the three steps are carried out in order, which then leads to the final proof of the theorem.
Due to the space limit, proofs of intermediate results are omitted.


\paragraph{Preliminaries}
We first introduce some notation. 
For any matrix $\bfA$, $\sp(\bfA)$ stands for the subspace spanned by the column vectors of $\bfA$. 
If we were given the oracle knowledge of $I = \supp(\bfU)$ and $J = \supp(\bfV)$, 
then we can define an oracle version of the observed matrix as
\begin{equation}
\label{eq:X-oracle}
\wt\bfX = (x_{ij}\Indc_{i\in I} \Indc_{j\in J}) \in \reals^{m\times n}.
\end{equation}
With appropriate rearrangement of rows and columns, the $I\times J$ submatrix concentrates on the top-left corner. 
From now on, we assume that this is the case.
We denote the singular value decomposition of $\bfX$ by 
\begin{align}
\label{eq:X-oracle-svd}
\wt\bfX = \begin{bmatrix}
\wt\bfU & \wt\bfU_\perp
\end{bmatrix}
\begin{bmatrix}
\wt\bfD & \bszero\\ \bszero & \wt\bfD_\perp
\end{bmatrix}
\begin{bmatrix}
\wt\bfV'\\ (\wt\bfV_\perp)'
\end{bmatrix},
\end{align}
where $\wt\bfU, \wt\bfD, \wt\bfV$ consist of the first $r$ singular triples of $\wt\bfX$, and $\wt\bfU_\perp, \wt\bfD_\perp, \wt\bfV_\perp$ contain the remaining $n-r$ triples (recall that we have assumed $m\geq n$).
In particular, the successive singular values of $\wt\bfX$ are denoted by $\wt{d}_1\geq \wt{d}_2 \geq \cdots \geq \wt{d}_n\geq 0$.


With the oracle knowledge of $I$ and $J$, we can define oracle versions of \prettyref{algo:init} and \prettyref{algo:IT}.
In the oracle version of \prettyref{algo:init}, we replace the subsets $I^0$ and $J^0$ by $\wt{I}^0 = I^0\cap I$ and $\wt{J}^0 = J^0\cap J$, and the output matrices are denoted by $\wt\bfU^{(0)}$ and $\wt\bfV^{(0)}$.
In the oracle version of \prettyref{algo:IT}, 
$\bfX$ is replaced by $\wt\bfX$ and 
$\bfV^{(0)}$ is replaced by $\wt\bfV^{(0)}$.
The intermediate matrices obtained after each step within the loop are denoted by $\wt\bfU^{(t),\mul}$, $\wt\bfU^{(t),\thr}$, $\wt\bfU^{(t)}$ and 
$\wt\bfV^{(t),\mul}$, $\wt\bfV^{(t),\thr}$, $\wt\bfV^{(t)}$, respectively.
We note that for any $t$, it is guaranteed that 
\begin{align*}
\supp(\wt\bfU^{(t),\thr}) &= \supp(\wt\bfU^{(t)})\subset I,\\
\supp(\wt\bfV^{(t),\thr}) &= \supp(\wt\bfV^{(t)})\subset J.
\end{align*}

To investigate the properties of the oracle sequence, we will trace the evolution of the columns subspaces of $\wt\bfU^{(t),\mul}$, $\wt\bfU^{(t)}$, $\wt\bfV^{(t),\mul}$ and $\wt\bfV^{(t)}$.
To this end, 
denote the $r$ canonical angles \citep{GolubLoan96} between $\sp(\wt\bfU^{(t),\mul})$ and $\sp(\wt\bfU)$ by $\pi/2\geq \phi_{u,1}^{(t)}\geq \cdots \geq \phi_{u,r}^{(t)}\geq 0$, and define 
\begin{align}
\sin\Phi_u^{(t)} = \diag(\sin\phi_{u,1}^{(t)},\dots, \sin\phi_{u,r}^{(t)}).
\end{align}
Moreover, denote the canonical angles between $\sp(\wt\bfU^{(t)})$ and $\sp(\wt\bfU)$ by $\pi/2\geq \theta_{u,1}^{(t)}\geq \cdots\geq \theta_{u,r}^{(t)}\geq 0$, and let 
\begin{align}
\sin\Theta_u^{(t)} = \diag(\sin\theta_{u,1}^{(t)},\dots, \sin\theta_{u,r}^{(t)}).
\end{align}
The quantities $\phi_{v,i}^{(t)}$, $\sin\Phi_v^{(t)}$, $\theta_{v,i}^{(t)}$ and $\sin\Theta_v^{(t)}$ are defined analogously.
For any pair of $m\times r$ orthonormal matrices $\bfW_1$ and $\bfW_2$, let the canonical angles between $\sp(\bfW_1)$ and $\sp(\bfW_2)$ be $\pi/2\geq \theta_1\geq \cdots\geq \theta_r\geq 0$ and $\sin\Theta = \diag(\sin\theta_1,\dots,\sin\theta_r)$, then \citep{Stewart90}
\begin{equation}
\begin{aligned}
\label{eq:sin-theta}
\fnorm{\sin\Theta} & = \frac{1}{\sqrt{2}}\fnorm{\bfW_1 \bfW_1' - \bfW_2 \bfW_2'},\\
\opnorm{\sin\Theta} & = \opnorm{\bfW_1 \bfW_1' - \bfW_2 \bfW_2'}.
\end{aligned}
\end{equation}

\paragraph{Overview} 
Given the oracle sequence defined as above, we divide the proof into three steps.
First, we show that the output of the oracle version of \prettyref{algo:init} gives a good initial value for the oracle version of \prettyref{algo:IT}.
Next, we prove two recursive inequalities that characterize the evolution of the column subspaces of $\wt\bfU^{(t)}$ and $\wt\bfV^{(t)}$, and show that after $T$ iterates, the output of the oracle version of \prettyref{algo:IT} estimates $\bfM$ well.
Last but not least, we show that with high probability the oracle estimating sequence and the actual estimating sequence are identical up to $3T$ iterates and that $\wh{T}\in [T, 3T]$. Therefore, the actual estimating sequence inherits all the nice properties that can be claimed for the oracle sequence.

In what follows, we carry out the three steps in order.

\paragraph{Initialization}
We first investigate the properties of $\wt\bfX$, $\wt{I}^0$, $\wt{J}^0$ and $\wt\bfV^{(0)}$.

Note that for any orthonormal matrix $\bfW$, $\bfW\bfW'$ gives the projection matrix onto $\sp(\bfW)$.
The following lemma quantifies the difference between the leading singular structures of $\bfX$ and $\bfM$.
\begin{lemma}
\label{lmm:oracle}
With probability at least $1 - m^{-2}$,
\begin{align}
	\label{eq:oracle-vec}
	\fnorm{\bfU\bfU - \wt\bfU \wt\bfU},\,
	\fnorm{\bfV\bfV - \wt\bfV \wt\bfV} 
	\leq \frac{\sqrt{2r}}{d_r}\pth{\sqrt{k} + \sqrt{l} + 2\sqrt{\log{m}}},	
\end{align}
and for any $i\in [n]$, 
\begin{align}
	\label{eq:oracle-val}
|\wt{d}_i - d_i|\leq \sqrt{k}+\sqrt{l}+2\sqrt{\log{m}} = o(d_r),
\end{align}
where the last equality holds under \prettyref{cond:asymp}.
\end{lemma}
\begin{proof}
By symmetry, we only need to spell out the arguments for $\bfU$ in \eqref{eq:oracle-vec}.
By definition, $\wt\bfX = \bfU\bfD\bfV' + \wt\bfZ$ where (after reordering of the rows and the columns) $\wt\bfZ = \begin{bmatrix}
\bfZ_{IJ} & \bszero \\ \bszero & \bszero
\end{bmatrix}$.
Thus, we have
\begin{align*}
\fnorm{\bfU\bfU' - \wt\bfU\wt\bfU'} \leq \sqrt{2r}\opnorm{\bfU\bfU' - \wt\bfU\wt\bfU'} \leq \frac{\sqrt{2r}}{d_r} \opnorm{\wt\bfZ}.
\end{align*}
Here, the first inequality holds since $\rank(\bfU\bfU' - \wt\bfU\wt\bfU')\leq 2r$ and the last inequality is due to Wedin's sin$\theta$ theorem \citep{Wedin72}.
By the Davidson-Szarek bound \citep{Davidson01}, with probability at least $1-m^{-2}$, $\opnorm{\wt\bfZ} = \opnorm{\bfZ_{IJ}} \leq \sqrt{k} + \sqrt{l} + 2\sqrt{\log{m}}$. 
This completes the proof of \eqref{eq:oracle-vec}.

On the other hand, Corollary 8.6.2 of \cite{GolubLoan96} implies that $|\wt{d}_i -d_i| \leq \opnorm{\bfZ_{IJ}}$. Together with the above discussion, we obtain the first inequality in \eqref{eq:oracle-val}.
The second inequality is a direct consequence of \prettyref{cond:asymp}.
This completes the proof.
\end{proof}

Next, we investigate the properties of the sets selected in \prettyref{algo:init}.
For some universal constants $0<a_- < 1 < a_+$, define the following two deterministic sets
\begin{align}
\label{eq:I0pm}
I^0_\pm = \sth{i\in [m]: \norm{\bfM_{i*}}^2 \geq a_{\mp}\alpha \sqrt{n\log{m}}},\quad
J^0_\pm = \sth{j\in [n]: \norm{\bfM_{*j}}^2 \geq a_{\mp}\alpha \sqrt{m \log{m}}}.
\end{align}


\begin{lemma}
\label{lmm:initial-set}
Let \prettyref{cond:asymp} be satisfied, and
let $\alpha \geq 4$, $a_-\leq \frac{1}{20}$ and $a_+\geq 2$ be fixed constants.
For sufficiently large values of $m$ and $n$,
with probability at least $1-O(m^{-2})$, 
we have $I_-\subseteq \wt{I}^0 \subseteq I_+$ and $J_-\subseteq \wt{J}^0 \subseteq J_+$, and so $I^0 = \wt{I}^0$ and $J^0 = \wt{J}^0$.
\end{lemma}
\begin{proof}
By symmetry, we only show the proof for $\wt{I}^0$ here. The arguments for $\wt{J}^0$ are similar.
On the one hand, we have
\begin{align*}
\Prob(I^0_-\nsubseteq \wt{I}^0) 
& \leq \sum_{i\in I^0_-} \Prob\pth{\norm{\bfX_{i*}}^2 < n + \alpha\sqrt{n\log{m}}} \\
& \leq m\, \Prob\pth{\chi^2_n(a_+\alpha\sqrt{n\log{m}}) < n + \alpha\sqrt{n\log{m}}} \\
& \leq m\, \exp\pth{-\frac{(a_+-1)^2 \alpha^2 n\log{m}}{4n + 8a_+\alpha\sqrt{n\log{m}}}}\\
& \leq m\, \exp(-3\log m) = m^{-2}.
\end{align*}
Here, the last inequality holds for fixed $a_+\geq 2$, $\alpha \geq 4$ and all sufficiently large $(m,n)$ such that $2a_+\alpha\sqrt{n\log{m}}\leq n/3$, which is guaranteed by \prettyref{cond:asymp}.

On the other hand, for $x = \frac{(1-a_-)^2 \alpha^2 n\log{m}}{(2.1)^2(n+2a_-\sqrt{n\log{m}})}$, we have
\begin{align*}
\Prob(\wt{I}^0\nsubseteq I_+^0) & \leq \sum_{i\in (I_+^0)^c}\Prob\pth{\norm{\bfX_{i*}}^2 >n+\alpha\sqrt{n\log m}}\\
& \leq m\, \Prob\pth{\chi^2_n(a_-\alpha\sqrt{n\log{m}}) > n + \alpha\sqrt{n\log{m}}} \\
& \leq m\, \Prob\pth{\chi^2_n(a_-\alpha\sqrt{n\log{m}}) > n + 2.1\sqrt{(n+2\alpha_-\sqrt{n\log{m}})\,x}}\\
& \leq m\, \Prob\pth{\chi^2_n(a_-\alpha\sqrt{n\log{m}}) > n + 2\sqrt{(n+2\alpha_-\sqrt{n\log{m}})x} + 2x}\\
& \leq m\, \exp\pth{-x}\\
& \leq m \exp(-3\log{m}) = m^{-2}.
\end{align*}
Here, the fourth inequality holds for fixed $\alpha \geq 4$, $a_-\leq \frac{1}{20}$, and all sufficiently large $(m,n)$ such that $n+2a_-\sqrt{n\log{m}} \geq \frac{2}{0.21}(1-a_-)^2\alpha\sqrt{n\log{m}}$. 
The last inequality holds when, in addition, $0.95^2 \cdot 16\cdot n \geq 3\cdot (2.1)^2\cdot (n+2a_-\sqrt{n\log{m}})$, which is again guaranteed by \prettyref{cond:asymp}.

Finally, when $I_-\subseteq \wt{I}^0 \subseteq I_+$, we have $I^0 = \wt{I}^0$ since $I_+\subset I$.
\end{proof}

The next lemma estimates the accuracy of the starting point $\wt\bfV^{(0)}$ for the oracle version of \prettyref{algo:IT}.

\begin{lemma}
\label{lmm:initial}
Let \prettyref{cond:asymp} be satisfied, and let $\alpha\geq 4$ and $a_+\geq 2$ be fixed constants. 
For sufficiently large values of $m$ and $n$, uniformly over $\calF(m,n,k,l,r,d,\kappa)$,
with probability at least $1 - O(m^{-2})$, for a positive constant $C$ that depends only on $\kappa, a_+$ and $\alpha$,
\[
\fnorm{\sin\wt\Theta_v^{(0)}} \leq 
\frac{C}{d}\qth{\pth{r^2 k^2 n\log{m}}^{1/4} + \pth{r^2 l^2 m\log{m}}^{1/4}}
\leq \frac{1}{6}.
\]
\end{lemma}
\begin{proof}
Let $\bfX^{(0)}$ be the matrix defined in Step \ref{algo:step:svd} of \prettyref{algo:init}, but with $I^0$ and $J^0$ replaced by $\wt{I}^0$ and $\wt{J}^0$.
Then we have
\[
\fnorm{\sin\wt\Theta_v^{(0)}} = \frac{1}{\sqrt{2}}\fnorm{\wt\bfV^{(0)}\wt\bfV^{(0)} - \wt\bfV\wt\bfV'}
\leq \frac{\sqrt{2r}}{\sqrt{2}}\opnorm{\wt\bfV^{(0)}\wt\bfV^{(0)} - \wt\bfV\wt\bfV'} 
\leq \frac{\sqrt{r}}{\wt{d}_r}\opnorm{\wt\bfX - \wt\bfX^{(0)}}.
\]
Here, the first equality is from \eqref{eq:sin-theta}. 
The second inequality holds since $\rank(\wt\bfV^{(0)}\wt\bfV^{(0)} - \wt\bfV\wt\bfV') \leq 2r$, and the last inequality is due to Wedin's sin$\theta$ theorem \citep{Wedin72}.

To further bound the rightmost side, we note that $\wt\bfX^{(0)}$ and $\wt\bfX$ are supported on $\wt{I}^0 \times \wt{J}^0$ and $I\times J$ respectively, with $\wt{I}^0 \times \wt{J}^0 \subset I\times J$.
In addition, $(I\times J)\backslash(\wt{I}^0 \times \wt{J}^0)$ is the union of two disjoint subsets $(I\backslash \wt{I}^0)\times J$ and $\wt{I}^0 \times (J\backslash \wt{J}^0)$.
Thus, the triangle inequality leads to
\begin{align}
\opnorm{\wt\bfX - \wt\bfX^{(0)}} & 
\leq \opnorm{\wt\bfX_{I\backslash \wt{I}^0, J}} + \opnorm{\wt\bfX_{\wt{I}^0, J\backslash \wt{J}^0}}  \nonumber \\
& \leq \opnorm{\bfU_{I\backslash \wt{I}^0, *}\bfD(\bfV_{J*})'} + 
\opnorm{\bfU_{\wt{I}^0 *} \bfD (\bfV_{J\backslash \wt{J}^0 *})'} + 
\opnorm{\bfZ_{I\backslash \wt{I}^0, J}} + 
\opnorm{\bfZ_{\wt{I}^0, J\backslash \wt{J}^0}}. 
\label{eq:initial-X-decomp}
\end{align}
We now bound each of the four terms in \eqref{eq:initial-X-decomp} separately. 
For the first term, on the event such that the conclusion of \prettyref{lmm:initial-set} holds, we have
\begin{align*}
\opnorm{\bfU_{I\backslash \wt{I}^0, *}\bfD\bfV_{J}'} \leq 
\opnorm{\bfD}\opnorm{\bfV_{J*}}\opnorm{\bfU_{I\backslash \wt{I}^0, *}}
\leq d_1 \fnorm{\bfU_{I\backslash \wt{I}^0, *}}
\leq \frac{d_1}{d_r} (a_+\alpha)^{1/2} (k^2 n\log{m})^{1/4}.
\end{align*}
Here, the last inequality is due to $I_-^0\subset \wt{I}^0$, the definition of $I_-^0$ in \eqref{eq:I0pm}, and the facts that $\norm{\bfM_{i*}}\geq d_r \norm{\bfU_{i*}}$ for all $i\in [m]$ and that $|I\backslash \wt{I}^0|\leq |I|\leq k$.
By similar argument, on the event such that the conclusion of \prettyref{lmm:initial-set} holds, we can bound the second term in \eqref{eq:initial-X-decomp} as
\begin{align*}
\opnorm{\bfU_{\wt{I}^0 *} \bfD (\bfV_{J\backslash \wt{J}^0 *})'} 
&\leq \opnorm{\bfU_{\wt{I}^0 *}}\opnorm{\bfD} \opnorm{\bfV_{J\backslash \wt{J}^0 *}} \leq d_1 \opnorm{\bfU} \fnorm{\bfV_{J\backslash \wt{J}^0 *}}\\
& \leq \frac{d_1}{d_r} (a_+\alpha)^{1/2} (l^2 m\log{m})^{1/4}.
\end{align*}
To bound the last two terms, we first note that on the event such that the conclusion of \prettyref{lmm:initial-set} holds, both terms are upper bounded by $\opnorm{\bfZ_{IJ}}$. 
Together with the Davidson--Szarek bound \citep{Davidson01}, this implies that with probability at least $1- m^{-2}$,
\[
\opnorm{\bfZ_{I\backslash \wt{I}^0, J}} + 
\opnorm{\bfZ_{\wt{I}^0, J\backslash \wt{J}^0}} \leq 2\opnorm{\bfZ_{IJ}} 
\leq 2\pth{\sqrt{k} + \sqrt{l} + 2\sqrt{\log{m}}}.
\]
Assembling the last five displays and observe that $\wt{d}_r \geq 0.9 d_r$ for sufficiently large values of $(m,n)$ on the event such that the conclusion of \prettyref{lmm:oracle}, we obtain the first inequality in the conclusion. 
The second inequality is a direct consequence of \prettyref{cond:asymp}.
This completes the proof.
\end{proof}

\paragraph{Evolution}
We now study how the column subspaces of $\wt\bfU^{(t)}$ and $\wt\bfV^{(t)}$ evolve over iterations.
To this end, let
\begin{equation}
\label{eq:rho}
\rho = \wt{d}_{r+1} / \wt{d}_r,
\end{equation}
where $\wt{d}_i$ denotes the $i\Th$ singular value of $\wt\bfX$.

\begin{proposition}
\label{prop:iter}
For any $t\geq 1$, let $x^t = \fnorm{\sin\Theta_u^{(t)}}$, $y^t = \fnorm{\sin\Theta_v^{(t)}}$. 
Moreover, define
\begin{align}
	\label{eq:omega}
\omega_u = (2\wt{d}_r)^{-1}\sqrt{k\gamma_u^2 }, \qquad 
\omega_v = (2\wt{d}_r)^{-1}\sqrt{l\gamma_v^2 },\qquad
\omega = \omega_u \vee \omega_v.
\end{align}
Let \prettyref{cond:asymp} be satisfied.
Then for sufficiently large values of $(m,n)$, on the event such that the conclusions of Lemmas \ref{lmm:oracle}--\ref{lmm:initial} hold,
\begin{enumerate}[1)]
\item 	For any $t\geq 1$, if $y^{t-1} < 1$, then
\begin{align}
	\label{eq:iter}
x^t \sqrt{1 - (y^{t-1})^2 } & \leq \rho y^{t-1} + \omega_u,\qquad
y^t \sqrt{1 - (x^t)^2} \leq \rho x^t + \omega_v.
\end{align}

\item For any $a\in (0,1/2]$, if
\begin{align}
	\label{eq:iter-conv}
	y^{t-1} \leq \frac{1.01\, \omega}{(1-a)(1-\rho)},
\end{align}
then so is $x^t$. Otherwise, 
\begin{align}
\label{eq:iter-contract}
x^t \leq y^{t-1}[1-a(1-\rho)].
\end{align}
The same conclusions hold with the ordered pair $(y^{t-1}, x^t)$ replaced by $(x^t, y^t)$ in \eqref{eq:iter-conv}--\eqref{eq:iter-contract}.
\end{enumerate}
\end{proposition}
\begin{proof}
1) In what follows, we focus on showing the first inequality in \eqref{eq:iter}. The second inequality follows from essentially the same argument.

Let $u^t = \fnorm{\sin\Phi_u^{(t)}}$. We first show that 
\begin{equation}
	\label{eq:iter-mul}
u^t \leq \frac{\rho y^{t-1}}{\sqrt{1-(y^{t-1})^2}}.	
\end{equation}
Recall the SVD of $\wt\bfX$ in \eqref{eq:X-oracle-svd}.
In addition, let the QR factorization of $\wt\bfU^{(t),\mul} = \wt\bfQ^{(t)} \wt\bfR^{(t),\mul}$.
By definition, $\wt\bfU^{(t),\mul} = \wt\bfX \wt\bfV^{(t-1)}$.
Premultiplying both sides by $\begin{bmatrix}
\wt\bfU & \wt\bfU_\perp
\end{bmatrix}'$, we obtain
\begin{align*}
\begin{bmatrix}\wt\bfD & \bszero\\ \bszero & \wt\bfD_\perp \end{bmatrix}
\begin{bmatrix}
\wt\bfV' \wt\bfV^{(t-1)} \\ (\wt\bfV_\perp)' \wt\bfV^{(t-1)}
\end{bmatrix}
= 
\begin{bmatrix}
\wt\bfU' \wt\bfQ^{(t)} \\
(\wt\bfU_\perp)' \wt\bfQ^{(t)}
\end{bmatrix} 
\wt\bfR^{(t),\mul}.
\end{align*} 
In addition, let
\begin{align*}
\begin{bmatrix} \wt\bfU' \wt\bfQ^{(t)} \\ (\wt\bfU_\perp)' \wt\bfQ^{(t)}
\end{bmatrix} 
= 
\begin{bmatrix}
\bfO^{(t)} \\ \bfW^{(t)}
\end{bmatrix}.
\end{align*}
By the last two displays, we have
\begin{align*}
\bfW^{(t)} & = \wt\bfD_\perp (\wt\bfV_{\perp})'\wt\bfV^{(t-1)} (\wt\bfR^{(t),\mul})^{-1}
 = \wt\bfD_\perp \qth{(\wt\bfV_{\perp})'\wt\bfV^{(t-1)}} 
\qth{\wt\bfV' \wt\bfV^{(t-1)}}^{-1} \wt\bfD^{-1} \qth{\wt\bfU' \wt\bfQ^{(t)}}.
\end{align*}
Thus, 
\begin{align*}
\fnorm{\bfW^{(t)} } \leq \opnorm{\wt\bfD_\perp} \fnorm{(\wt\bfV_{\perp})'\wt\bfV^{(t-1)}} 
\opnorm{[\wt\bfV' \wt\bfV^{(t-1)}]^{-1}}
\opnorm{\wt\bfD^{-1}}
\opnorm{\wt\bfU} \opnorm{\wt\bfQ^{(t)}}.
\end{align*}
By Corollary 5.5.4 of \cite{Stewart90}, $\fnorm{\bfW^{(t)} } = u^t$, $\fnorm{(\wt\bfV_{\perp})'\wt\bfV^{(t-1)}} = y^{t-1}$. 
Moreover, by Section 12.4.3 of \cite{GolubLoan96}, $\opnorm{[\wt\bfV' \wt\bfV^{(t-1)}]^{-1}} = 1/\cos\theta_{v,r}^{(t-1)} = 1/\sqrt{1-(\sin\theta_{v,r}^{(t-1)})^2} \leq 1/\sqrt{1-(y^{t-1})^2}$. 
Here we have used the assumption that $y^{t-1}<1$.
Together with the facts that $\opnorm{\wt\bfD_\perp}=\wt{d}_{r+1}$, $\opnorm{\wt{\bfD}^{-1}} = \wt{d}_r^{-1}$, $\opnorm{\wt\bfU} = \opnorm{\wt\bfQ^{(t)}} = 1$, this leads to \eqref{eq:iter-mul}.

Next, we show that
\begin{equation}
\label{eq:iter-thr}
x^t \leq u^t + \frac{\omega_u}{\sqrt{1-(y^{t-1})^2}}.
\end{equation}
To this end, let 
$w^t = \fnorm{\wt\bfQ^{(t)}(\wt\bfQ^{(t)})' - \wt\bfU^{(t)}(\wt\bfU^{(t)})'}$.
Then, by \eqref{eq:sin-theta} and the triangle inequality, we obtain
\[
x^t\leq u^t + \frac{1}{\sqrt{2}}w^t.
\] 
To bound $w^t$, note that Wedin's sin$\theta$ theorem \citep{Wedin72} implies 
\begin{align*}
w^t \leq \frac{\fnorm{\wt\bfU^{(t),\mul} - \wt\bfU^{(t)}}}{\sigma_r(\wt\bfU^{(t),\mul})}.
\end{align*}
In the oracle version, $\wt\bfU^{(t),\mul}$ has at most $k$ nonzero rows, and so $\fnorm{\wt\bfU^{(t),\mul} - \wt\bfU^{(t)}} \leq \sqrt{k\gamma_u^2}$.
For any unit vector $\bfy\in \sp(\wt\bfV^{(t-1)})$, decompose $\bfy = \bfy_0 + \bfy_1$ where $\bfy_0\in \sp(\wt\bfV)$ and $\bfy_1\in \sp(\wt\bfV_\perp)$.
Then by definition, $\norm{\bfy_0} \geq \cos\theta_{v,1}^{(t-1)} \geq \sqrt{1-(y^{t-1})^2}$.
Thus, for any unit vector $\bfx$, 
$\norm{\wt\bfU^{(t),\mul}\bfx}^2 = \norm{\wt\bfX \bfV^{(t-1)}\bfx}^2 = \norm{\wt\bfX \bfy}^2 = \norm{\wt\bfX \bfy_0}^2 + \norm{\wt\bfX \bfy_1}^2 \geq \norm{\wt\bfX \bfy_0}^2 = \norm{\wt\bfX\wt\bfV \wt\bfV'\bfy_0}^2 \geq (\wt{d}_r)^2 \norm{\bfy_0}^2 \geq (\wt{d}_r)^2 [1 - (y^{t-1})^2]$.   
Hence,
\[
\sigma_r(\wt\bfU^{(t),\mul}) \geq \inf_{\norm{\bfx}=1}\norm{\wt\bfU^{(t),\mul}\bfx}
\geq \wt{d}_r \sqrt{1 - (y^{t-1})^2}.
\]
Assembling the last three display, we obtain \eqref{eq:iter-thr}.
Finally, the first inequality in \eqref{eq:iter} comes from \eqref{eq:iter-mul}, \eqref{eq:iter-thr} and the triangle inequality.

2) Given \eqref{eq:iter}, we have 
\[
x^t \leq \frac{\rho y^{t-1} + \omega}{\sqrt{1-(y^{t-1})^2}},
\]
and that $y^0\leq \frac{1}{6} \leq \frac{1}{5}(1-\rho)^2$ for sufficiently large values of $(m,n)$ due to \prettyref{cond:asymp} and \prettyref{lmm:oracle}.
The proof of part (2) then follows from the same argument as in the proof of Proposition 6.1 in \cite{Ma11}. 
\end{proof}

\paragraph{Convergence}

We say that the oracle sequence has \emph{converged} if 
\begin{align}
\label{eq:converge}
x^t \vee y^t \leq \frac{1.01 \omega}{(1-m^{-1})(1-\rho)}.
\end{align}
This choice is motivated by the observation that $\frac{1.01\omega}{1-\rho}$ is the smallest possible value for $x^t$ and $y^t$ that \prettyref{prop:iter} can lead to.
\begin{proposition}
\label{prop:converge}
Let \prettyref{cond:asymp} be satisfied and $T$ be defined in \eqref{eq:T}.
For sufficiently large values of $(m,n)$, on the event such that the conclusions of Lemmas \ref{lmm:oracle}--\ref{lmm:initial} hold, it takes at most $T$ steps for the oracle sequence to converge in the sense of \eqref{eq:converge}. 
For any $t$, let $\wt\bfP_u^{(t)} = \wt\bfU^{(t)}(\wt\bfU^{(t)})'$ and
$\wt\bfP_v^{(t)} = \wt\bfV^{(t)}(\wt\bfV^{(t)})'$.
Then there exists a constant $C$ that depends only on $\kappa$, such that for all $t\geq T$,
\[
\fnorm{\wt\bfP_u^{(t)}\wt\bfX \wt\bfP_v^{(t)} - \wt\bfU \wt\bfD \wt\bfV}^2  \leq C\pth{k\gamma_u^2 + l\gamma_v^2}.
\]
\end{proposition}
\begin{proof}
To prove the first claim, we rely on claim (2) of \prettyref{prop:iter}.
Without loss of generality, assume that $m = 2^\nu$ for some integer $\nu \geq 1$. 
So $\nu = \log{m}/\log{2}$.
Let $t_1$ be the number of iterations needed to ensure that $x^t\vee y^t \leq \frac{1.01 \omega}{(1-\frac{1}{2})(1-\rho)}$. 
Note that when \eqref{eq:iter-conv} does not hold, \eqref{eq:iter-contract} ensures that 
\begin{equation}
	\label{eq:iter-decay}
y^t \leq y^{t-1} [1-a(1-\rho)]^2, \quad x^t \leq x^{t-1} [1-a(1-\rho)]^2.	
\end{equation}
Thus, it suffices to have 
$\qth{1-\frac{1}{2}(1-\rho)}^{2t_1} \geq \frac{1.01 \omega}{(1-\frac{1}{2})(1-\rho)}$,
\ie, $2t_1 |\log(1-\frac{1}{2}(1-\rho))| \geq \log(1-\frac{1}{2})(1-\rho)/(1.01\omega)$. 
Since $|\log(1-x)|\geq x$ for all $x\in (0,1)$, it suffices to set
\begin{align*}
t_1 = \frac{1}{1-\rho}\log\frac{\frac{1}{2}(1-\rho)}{1.01\omega} = \frac{1+o(1)}{2}\log\pth{\frac{d_r^2}{k\gamma_u^2 \vee l\gamma_v^2}}.
\end{align*}
Next, let $t_2-t_1$ be the number of additional iterations needed to achieve $x^t\vee y^t \leq 1.01\omega / [(1-\frac{1}{4})(1-\rho)]^2$.
Before this is achieved, \eqref{eq:iter-decay} is satisfied with $a = \frac{1}{4}$.
So it suffices to have $[1-\frac{1}{4}(1-\rho)]^{2(t_2 - t_1)}\leq (1-\frac{1}{2})/(1-\frac{1}{4})$, which is guaranteed if
$t_2-t_1 \geq \frac{2}{1-\rho}[\log(1-\frac{1}{4})-\log(1-\frac{1}{2})]$.
Recursively, we define $t_i$ for $i=3,\dots,\nu$, such that $x^{t_i}, y^{t_i}\leq 1.01\omega / [(1-2^{-i})(1-\rho)]$.
Repeating the above argument shows that it suffices to have
$t_i-t_{i-1} = \frac{2^{i-1}}{1-\rho}[\log(1-2^{-i}) -\log(1-2^{-(i-1)})]$ for $i=3,\dots,\nu$.
Therefore, if we let
\begin{align*}
t_\nu - t_1 = \frac{\nu+1/2}{2(1-\rho)} = \frac{(1+o(1))\log m}{2\log 2}
\geq \sum_{i=1}^\nu \frac{2^{i-1}}{1-\rho}\qth{\log(1-2^{-i}) - \log(1-2^{-(i-1)})},
\end{align*}
then $x^t\vee y^t \leq 1.01\omega/[(1-m^{-1})(1-\rho)]$ for all $t \geq t_\nu$.
We complete the proof of the first claim by noting that $T\geq t_\nu$ for sufficiently large $m$, $n$ under \prettyref{cond:asymp}.
	
To prove the second claim, let $\wt\bfP_u = \wt\bfU\wt\bfU'$ and $\wt\bfP_v = \wt\bfV\wt\bfV'$. 
Then we have
\begin{align}
\fnorm{\wt\bfP_u^{(t)}\wt\bfX \wt\bfP_v^{(t)} - \wt\bfU \wt\bfD \wt\bfV}
& = \fnorm{\wt\bfP_u^{(t)}\wt\bfX \wt\bfP_v^{(t)} - \wt\bfP_u \wt\bfX \wt\bfP_v} \label{eq:conv-bd-1}\\
& \leq \fnorm{(\wt\bfP_u^{(t)} - \wt\bfP_u)\wt\bfX \wt\bfP_v^{(t)}}
+ \fnorm{\wt\bfP_u \wt\bfX (\wt\bfP_v^{(t)} - \wt\bfP_v)} \nonumber\\
& \leq \fnorm{\wt\bfP_u^{(t)} - \wt\bfP_u} \opnorm{\wt\bfX} \opnorm{\wt\bfP_v^{(t)}} + \fnorm{\wt\bfP_v^{(t)} - \wt\bfP_v} \opnorm{\wt\bfX} \opnorm{\wt\bfP_u} \nonumber \\
& = \wt{d}_1 \pth{\fnorm{\wt\bfP_u^{(t)} - \wt\bfP_u} + \fnorm{\wt\bfP_v^{(t)} - \wt\bfP_v}} \label{eq:conv-bd-2}\\
& \leq C\pth{\sqrt{k\gamma_u^2} + \sqrt{l\gamma_v^2}}. \label{eq:conv-bd-3}
\end{align}
Here, the equality \eqref{eq:conv-bd-1} is due to the definitions of $\wt\bfP_u$, $\wt\bfP_v$ and the fact that $\wt\bfU$, $\wt\bfD$ and $\wt\bfV$ consist of the first $r$ singular values and vectors of $\wt\bfX$. 
The equality \eqref{eq:conv-bd-2} holds since $\opnorm{\wt\bfX} = \wt{d}_1$ and $\opnorm{\wt\bfP_u} = \opnorm{\wt\bfP_v^{(t)}} = 1$ as both are projection matrices.
Finally, the inequality \eqref{eq:conv-bd-3} holds since $\fnorm{\wt\bfP_u^{(t)} - \wt\bfP_u} = \sqrt{2} x^t$ and $\fnorm{\wt\bfP_v^{(t)} - \wt\bfP_v} = \sqrt{2} y^t$ due to \eqref{eq:sin-theta}, the definitions in \eqref{eq:omega} and \eqref{eq:converge}, and the fact that on the event such that \eqref{eq:oracle-val} holds, $\wt{d}_1 / \wt{d}_r \leq 2 \kappa$ when $m$ and $n$ are sufficiently large. 
This completes the proof.
\end{proof}

\begin{remark}
It is worth noting that the conclusions of \prettyref{prop:iter} and \prettyref{prop:converge} hold for any $\gamma_u > 0$ and $\gamma_v > 0$, though they will be used later with the specific choice of $\gamma_u$ and $\gamma_v$ in \eqref{eq:thr}.
\end{remark}

\paragraph{Proof of Upper Bounds}

We are now in the position to prove \prettyref{thm:upper}. To this end, we need to establish the equivalence between the oracle and the actual estimating sequences.
The following lemma shows that with high probability, the oracle sequence and the actual sequence are identical up to $3T$ iterates.

\begin{lemma}
\label{lmm:oracle-id}
Let $\gamma_u$ and $\gamma_v$ be defined as in \eqref{eq:thr} with some fixed constant $\beta\geq 4$ and let \prettyref{cond:asymp} be satisfied.
For sufficiently large $m$ and $n$, with probability at least $1-O(m^{-2})$, for all $1\leq t\leq 3T$, 
$\bfU^{(t)}_{I^c *} = \bszero$, $\bfV^{(t)}_{J^c *} = \bszero$, 
and so $\bfU^{(t)} = \wt\bfU^{(t)}$ and $\bfV^{(t)} = \wt\bfV^{(t)}$.
\end{lemma}
\begin{proof}
First of all, by \prettyref{lmm:initial-set}, with probability at least $1-O(m^{-2})$, $J^0 = \wt{J}^0 \subset J_+ \subset J$, and so $\wt\bfV^{(0)} = \bfV^{(0)}$. Define event $E^{(0)} = \{\bfV^{(0)} = \wt\bfV^{(0)}\}$.

We now focus on the first iteration. Define event
\begin{align*}
E_u^{(1)} = \sth{\norm{\bfZ_{i*} \wt\bfV^{(0)}} < \gamma_u,\, \forall i\in I^c}.
\end{align*} 
On $E^{(0)}\cap E_u^{(1)}$, for any $i\in I^c$, 
$\bfU^{(1),\mul}_{i*} = \bfX_{i*} \bfV^{(0)} = \bfZ_{i*} \wt\bfV^{(0)}$.
Thus, $\norm{\bfU^{(1),\mul}_{i*}} < \gamma_u $ and so $\bfU^{(1),\thr}_{i*} = \bszero$ for all $i\in I^c$. 
This further implies $\bfU^{(1)}_{I^c *} = \bszero$ and $\bfU^{(1)} = \wt\bfU^{(1)}$.
Further define event
\begin{align*}
E_v^{(1)} = \sth{\norm{(\bfZ_{*j})' \wt\bfU^{(1)}} < \gamma_v,\,\forall j\in J^c }.
\end{align*}
Then by similar argument, on the event $E^{0}\cap E_u^{(1)}\cap E_v^{(1)}$,
we have $\bfV^{(1)}_{J^c *} = \bszero$ and $\bfV^{(1)} = \wt\bfV^{(1)}$.

We now bound the probability of $(E_u^{(1)})^c$.
Without loss of generality, let $J\subset [l]$. 
Note that for any $j\in J$, $i\in I^c$, $\wt\bfV^{(0)}$ depends on $Z_{ij}$ only through $\norm{\bfZ_{I^c j}}^2$ in the selection of $\wt{J}^0$ in the oracle version of \prettyref{algo:init}.
Therefore, $\wt\bfV^{(0)}$ is independent of $\frac{Z_{ij}}{\norm{\bfZ_{I^c j}}}$. 
Let $k'=|I^c|$ and $Y_1,\dots Y_l$ be i.i.d.~$\chi_{k'}$ random variables independent of $\bfZ$.
For any $i\in I^c$ and $j\in [l]$, let
\begin{align*}
	\check{Z}_{ij} = Y_j \frac{Z_{ij}}{\norm{\bfZ_{I^c j}}}, 
\end{align*}
and $\check\bfZ_{i[l]} = (\check{Z}_{i1},\dots, \check{Z}_{il}) \in \reals^{1\times l}$.
Since $\supp(\wt\bfV^{(0)}) \subset J\subset [l]$ on the event $E^{(0)}$, 
we obtain that for any $i\in I^c$, 
$\bfZ_{i*}\wt\bfV^{(0)} = \bfZ_{i[l]}\wt\bfV^{(0)} = \check\bfZ_{i[l]}\wt\bfV^{(0)} + (\bfZ_{i[l]} - \check\bfZ_{i[l]})\wt\bfV^{(0)}$. 
Thus,
\begin{align*}
\norm{\bfZ_{i*}\wt\bfV^{(0)}} \leq 
\norm{\check\bfZ_{i[l]}\wt\bfV^{(0)}_{[l]*}} + 
\norm{(\bfZ_{i[l]} - \check\bfZ_{i[l]})\wt\bfV^{(0)}_{[l]*}}
\leq 
\norm{\check\bfZ_{i[l]}\wt\bfV^{(0)}_{[l]*}} + 
\norm{\bfZ_{i[l]} - \check\bfZ_{i[l]}} \opnorm{\wt\bfV^{(0)}_{[l]*}}.
\end{align*}
For the first term on the rightmost side, since $\check\bfZ_{i[l]}$ is independent of $\wt\bfV^{(0)}$, $\norm{\check\bfZ_{i[l]}\wt\bfV^{(0)}_{[l]*}}^2\sim \chi^2_r$, and so by \prettyref{lmm:chisq}, with probability at least $1-O(m^{-\beta})$,
\begin{align*}
\norm{\check\bfZ_{i[l]}\wt\bfV^{(0)}_{[l]*}}^2 
\leq r + 2\sqrt{\beta r\log{m}}+2\beta\log{m}.
\end{align*}
For the second term, we first note that $\opnorm{\wt\bfV^{(0)}_{[l]*}} = 1$ since it has orthonormal columns.
Moreover, $\check\bfZ_{i[l]} - \bfZ_{i[l]} = \bfZ_{i[l]}\diag\pth{\frac{Y_1}{\norm{\bfZ_{I^c 1}}}-1,\dots,\frac{Y_1}{\norm{\bfZ_{I^c l}}}-1}$.
Thus,
\begin{align*}
\norm{\check\bfZ_{i[l]} - \bfZ_{i[l]}} \leq \norm{\bfZ_{i[l]}} \max_{j\in [l]} \left| \frac{Y_j}{\norm{\bfZ_{I^c j}}}-1 \right|.
\end{align*}
By \prettyref{lmm:chisq}, with probability at least $1-O(m^{-\beta})$,
\begin{align*}
\norm{\bfZ_{i[l]}}^2 \leq l + 2\sqrt{\beta l\log{m}} + 2\beta \log{m}.
\end{align*}
By \prettyref{lmm:chisq-ratio}, for any $j\in [l]$, with probability at least $1 - O(m^{-(\beta+1)})$,
\begin{align*}
\left| \frac{Y_j}{\norm{\bfZ_{I^c j}}}-1 \right|
\leq \left| \frac{Y_j^2}{\norm{\bfZ_{I^c j}}^2}-1 \right| 
\leq 4\cdot 1.01 \cdot \sqrt{\frac{(\beta+1)\log{m}}{k'}}.
\end{align*}
Here, the last inequality holds 
for sufficient large values of $m$ and $n$,
since \prettyref{cond:asymp} implies that $(\log{m})/k' = o(1)$.
By the union bound, with probability at least $1-O(m^{-\beta})$, for sufficient large values of $m$ and $n$,
\begin{align*}
\norm{\check\bfZ_{i[l]} - \bfZ_{i[l]}} \leq 0.01\sqrt{\log{m}}\, ,
\end{align*}
since \prettyref{cond:asymp} ensures that $l/k'=o(1)$.
Assembling the last six displays, we obtain that for any $\beta > 1$, with probability at least $1-O(m^{-\beta})$,
\begin{align*}
\norm{\bfZ_{i*}\wt\bfV^{(0)}} \leq \sqrt{r + 2\sqrt{\beta r\log{m}}+2\beta\log{m}} + 
0.01\sqrt{\log{m}}
\leq \gamma_v.
\end{align*}
Applying the union bound again, we obtain that when $\beta \geq 4$ in \eqref{eq:thr},
\begin{align}
\label{eq:iter-thr-event}
\Prob\sth{(E^{(1)}_u)^c} \leq O(m^{-3}).
\end{align}
Similarly, for any $j\in J^c$, $\wt\bfU^{(1)}$ depends on $Z_{ij}$ only through $\norm{\bfZ_{i J^c}}$.
Therefore, by analogous arguments, we also obtain \eqref{eq:iter-thr-event} for $(E_v^{1})^c$ with any fixed $\beta \geq 4$.

Turn to subsequent iterations, we further define events
\begin{align*}
E_u^{(t)} = \sth{\norm{\bfZ_{i*} \wt\bfV^{(t-1)}} < \gamma_u,\, \forall i\in I^c},
\quad 
E_v^{(t)} = \sth{\norm{(\bfZ_{*j})' \wt\bfU^{(t)}} < \gamma_v,\,\forall j\in J^c },
\quad t = 2,\dots, 3T.
\end{align*}
Iterating the above arguments, we obtain that on the event $E^{(0)}\cap (\cap_{t=1}^{3T} E_u^{(t)}) \cap (\cap_{t=1}^{3T} E_v^{(t)})$, 
$\bfU^{(t)}_{I^c *} = \bszero$, $\bfV^{(t)}_{J^c *} = \bszero$, 
and so $\bfU^{(t)} = \wt\bfU^{(t)}$ and $\bfV^{(t)} = \wt\bfV^{(t)}$.
Moreover, by similar argument to that for \eqref{eq:iter-thr-event}, we can bound each $\Prob\{(E_u^{(t)})^c\}$ and $\Prob\{(E_v^{(t)})^c\}$ by $O(m^{-3})$ for all $t=2,\dots, 3T$ with any fixed $\beta\geq 4$ in \eqref{eq:thr}.
Finally, under \prettyref{cond:asymp}, $T = O(m)$, and so
\[
\Prob\sth{E^{(0)}\cap (\cap_{t=1}^{3T} E_u^{(t)}) \cap (\cap_{t=1}^{3T} E_v^{(t)})} = 1 - O(m^{-2}).
\]
This completes the proof.
\end{proof}

\begin{lemma}
\label{lmm:stop}
Let $\wh{T}$ be defined in \eqref{eq:whT}. With probability at least $1-O(m^{-2})$, $T\leq \wh{T} \leq 3T$.
\end{lemma}
\begin{proof}
By definition \eqref{eq:thr} and \eqref{eq:T}, we have
\begin{align*}
T \leq \frac{1.01}{2} \pth{\frac{\log{m}}{\log{2}} + \log\frac{d_r^2}{\gamma^2}}.
\end{align*}
On the other hand, note that $1/\log{2} \geq 1.44$ and that $\log(k\vee l)\leq \log{m}$ under the assumption that $m\geq n$, and hence
\begin{align*}
T \geq \frac{1.01}{2} \pth{\frac{\log{m}}{\log{2}} - \log{m} + \log\frac{d_r^2}{\gamma^2}}
\geq \frac{1.01}{2}\pth{0.44 \log{m} + \log\frac{d_r^2}{\gamma^2}}.
\end{align*}
On the other hand, on the event such that the conclusions of Lemmas \ref{lmm:oracle}--\ref{lmm:initial} hold, we have
\begin{align*}
|d_r^{(0)} -d_r| & \leq |d_r^{(0)} - \wt{d}_r| + |\wt{d}_r - d_r| \\
& = |\wt{d}_r^{(0)} - \wt{d}_r| + |\wt{d}_r - d_r|\\
& \leq \opnorm{\wt{\bfX}^{(0)} - \wt{\bfX}} + o(d_r)\\
& = o(d_r).
\end{align*}
Hence for sufficiently large values of $m$ and $n$, $\log\gamma^2 > 1$ and with probability at least $1-O(m^{-2})$, $|\log(d_r^{(0)})^2 /\log d_r^2 - 1|\leq 0.01$.
When the above inequalities all hold, we obtain $\wh{T}\in [T, 3T]$.
\end{proof}

We are now in the position to prove \prettyref{thm:upper}.
\begin{proof}[Proof of \prettyref{thm:upper}]
	Note that on the events such that the conclusions of Lemmas \ref{lmm:oracle}--\ref{lmm:stop} hold, we have
	\begin{align*}
	& \hskip -1em \fnorm{\wh{\bfM} - \bfM}\\ 
	& = \fnorm{\wh\bfP_u^{(\wh{T})} \bfX \wh\bfP_v^{(\wh{T})} - \bfU\bfD\bfV'}\\
	& \leq \fnorm{\wh\bfP_u^{(\wh{T})} \bfX \wh\bfP_v^{(\wh{T})} - \wt\bfU\wt\bfD\wt\bfV'} + \fnorm{\wt\bfU\wt\bfD\wt\bfV' - \bfU\bfD\bfV'}\\
	& = \fnorm{\wt\bfP_u^{(\wh{T})} \wt\bfX \wt\bfP_v^{(\wh{T})} - \wt\bfP_u \wt\bfX \wt\bfP_v} 
	+ \fnorm{\wt\bfP_u\wt\bfX \wt\bfP_v - \bfP_u \bfM\bfP_v}\\
	& \leq \fnorm{\wt\bfP_u^{(\wh{T})} \wt\bfX \wt\bfP_v^{(\wh{T})} - \wt\bfP_u \wt\bfX \wt\bfP_v} 
	+ \fnorm{\wt\bfP_u\wt\bfX \wt\bfP_v - \wt\bfP_u \bfM \wt\bfP_v}\\
	& \quad
	+ \fnorm{\wt\bfP_u \bfM \wt\bfP_v - \bfP_u \bfM \bfP_v}.
	\end{align*}	
	Here, the first and the second inequalities are both due to the triangle inequality.
	The second equality is due to \prettyref{lmm:stop} and the facts that $\supp(\wt\bfU^{(t)})\subset I$, $\supp(\wt\bfV^{(t)})\subset J$ and that $\wt\bfU$ and $\wt\bfV$ collect the first $r$ left and right singular vectors of $\wt\bfX$.

	We now bound each of the three terms on the rightmost side of the last display.
	First, on the event such that the conclusions of \prettyref{prop:converge} and \prettyref{lmm:stop} hold, we have
	\begin{align*}
	\fnorm{\wt\bfP_u^{(\wh{T})} \wt\bfX \wt\bfP_v^{(\wh{T})} - \wt\bfP_u \wt\bfX \wt\bfP_v} \leq C\sqrt{k\gamma_u^2 + l\gamma_v^2}.
	\end{align*}
	Next, by similar argument to that leading to the conclusion of \prettyref{lmm:initial}, with probability at least $1-O(m^{-2})$
	\begin{align*}
	\fnorm{\wt\bfP_u\wt\bfX \wt\bfP_v - \wt\bfP_u \bfM \wt\bfP_v}
	& \leq \fnorm{\wt\bfX  - \bfM } = \fnorm{\bfZ_{IJ}} \\
	& \leq \sqrt{r}\opnorm{\bfZ_{IJ}} \\
	& \leq \sqrt{r}(\sqrt{k}+\sqrt{l}+2\sqrt{\log{m}}).
	\end{align*}
	Last but not least, 
	\begin{align*}
	& \hskip -1em \fnorm{\wt\bfP_u \bfM \wt\bfP_v - \bfP_u \bfM \bfP_v} \\
	& \leq \fnorm{(\wt\bfP_u - \bfP_u)\bfM \wt\bfP_v} + 
	\fnorm{\bfP_u\bfM(\wt\bfP_v - \bfP_v)}\\
	& \leq d_1 (\fnorm{\wt\bfP_u - \bfP_u} + \fnorm{\wt\bfP_v - \bfP_v})\\
	& \leq \kappa \sqrt{r}(\sqrt{k}+\sqrt{l}+2\sqrt{\log{m}}).
	\end{align*}
	Assembling the last four displays, we complete the proof for the case of Frobenius norm, \ie, $q=2$.
	To obtain the result for all $q\in [1,2)$, simply note that for any matrix $\bfA$, $\sqnorm{\bfA} \leq (\rank(\bfA))^{\frac{1}{q}-\frac{1}{2}}\fnorm{\bfA}$ and that $\rank(\wt\bfM - \bfM)\leq 2r$. This completes the proof.
\end{proof}

\subsection{Proof of \prettyref{prop:r-hat}}
\label{sec:proof-r-hat}
\begin{proof}[Proof of \prettyref{prop:r-hat}]
Without loss of generality, assume that $\sigma = 1$.
We first show that $\wh{r}\leq r$ with probability at least $1-O(m^{-2})$.
To this end, note that
\begin{align*}
\Prob\sth{\wh{r} > r} & = \Prob\sth{\sigma_{r+1}(\bfX_{I^0 J^0}) > \delta_{|I^0| |J^0|}}\\
& \leq \Prob\sth{\max_{|A| = |I^0|, |B| = |J^0|}\sigma_{r+1}(\bfX_{A B}) > \delta_{|A| |B|}}\\
& \leq \sum_{i=r+1}^m \sum_{j=r+1}^n \Prob\sth{\max_{|A| = i, |B|=j} \sigma_{r+1}(\bfX_{AB}) > \delta_{ij}}.
\end{align*}
By the interlacing property of singular values, we know that for $\bfZ$, a $m\times n$ standard Gaussian random matrix, 
\[
\max_{|A| = i, |B|=j} \sigma_{r+1}(\bfX_{AB}) \stackrel{st}{<}
\max_{|A| = i-r, |B|=j-r} \sigma_{1}(\bfZ_{AB}) 
\stackrel{st}{<} 
\max_{|A| = i, |B|=j} \sigma_{1}(\bfZ_{AB}),
\]
where $\stackrel{st}{<}$ means stochastically smaller.
Together with the union bound, this implies
\begin{align*}
\Prob\sth{\max_{|A| = i, |B|=j} \sigma_{r+1}(\bfX_{AB}) > \delta_{ij}}
& \leq {m\choose i}{n\choose j}\Prob\sth{\sigma_1(\bfZ_{AB}) > \delta_{ij}}\\
& \leq \pth{\eexp m\over i}^i \pth{\eexp n\over j}^j
\exp\pth{-i\log\frac{\eexp m}{i} - j\log\frac{\eexp m}{j} - 4\log m }\\
& = m^{-4}.
\end{align*}
Here, the second inequality is due to ${p\choose k}\leq (\eexp p/k)^k$ for any $k\in [p]$ and the Davidson-Szarek bound \citep{Davidson01}.
As $n\leq m$ under \prettyref{cond:asymp}, we obtain
\begin{align*}
\Prob\sth{\wh{r} > r} \leq 	\sum_{i=r+1}^m \sum_{j=r+1}^n m^{-4} \leq m^{-2}.
\end{align*}

To show that $\wh{r} \geq r$ with probability at least $1-O(m^{-2})$, we note that on the event such that the conclusions of Lemmas \ref{lmm:oracle}--\ref{lmm:initial} hold, $\sigma_r(\bfX_{I^0 J^0}) = \sigma_r(\wt\bfX^0) = \wt{d}_r^{(0)}$.
So by the triangle inequality, the conclusion of \prettyref{lmm:oracle} and the proof of \prettyref{lmm:initial},
we obtain that
\begin{align*}
\sigma_r(\bfX_{I^0 J^0}) = \wt{d}_r^{(0)} \geq d_r - \opnorm{\bfX - \wt\bfX} - \opnorm{\wt\bfX - \wt\bfX^0}
\geq d_r/4 > \delta_{kl},
\end{align*}
where the second last and the last inequalities hold under \prettyref{cond:asymp} for sufficiently large values of $m$ and $n$.
Note that on the event such that the conclusion of \prettyref{lmm:initial-set} holds, we have $|I^0|\leq k$ and $|J^0|\leq l$ and so $\delta_{|I^0| |J^0|}\leq \delta_{kl}$.
This completes the proof.
\end{proof}

%% file: appendix.tex

\appendix

\section{Appendix}
\label{app:appendix}

\begin{lemma}[Lemma 8.1 in \cite{Birge01Alternative}]
	\label{lmm:chisq}
Let $X$ follow the non-central chi square distribution $\chi^2_\nu(\delta)$ with degrees of freedom $\nu$ and non-centrality parameter $\delta \geq 0$. 
Then for any $x>0$, 
\begin{align*}
& \Prob\sth{X\geq \nu+\delta + 2\sqrt{(\nu + 2\delta)x} + 2x} \leq \eexp^{-x},\\
& \Prob\sth{X\leq \nu+\delta - 2\sqrt{(\nu + 2\delta)x}}  \leq \eexp^{-x}.
\end{align*}
\end{lemma}

\begin{lemma}
\label{lmm:chisq-ratio}
Let $X$ and $Y$ be two independent $\chi^2_\nu$ random variables. Then for any $x > 0$, 
\begin{align*}
\Prob\sth{\left| \frac{X}{Y}-1 \right| \leq \frac{4\sqrt{ x/\nu} (1+\sqrt{ x/\nu}  )}{1 - 2\sqrt{x/\nu}} } \geq 1 - 4\eexp^{-x}.
\end{align*}
\end{lemma}
\begin{proof}
By the triangle inequality,
\begin{align*}
\left| \frac{X}{Y}-1 \right|\leq \frac{1}{|Y|} \pth{|X-\nu| + |Y-\nu|}.
\end{align*}
By \prettyref{lmm:chisq}, for any $x>0$, each of the following holds with probability at least $1-2\eexp^{-x}$:
\begin{align*}
& |X-\nu| \leq 2\sqrt{\nu x} + 2x,\\
& |Y-\nu| \leq 2\sqrt{\nu x} + 2x, \quad \mbox{and}\quad |Y|\geq \nu - 2\sqrt{\nu x}.
\end{align*}
Assembling the last two displays, we complete the proof.
\end{proof}